\documentclass[11pt]{amsart}

\usepackage[a4paper,margin=1in]{geometry}
\usepackage{amsmath,amssymb,amsthm}
\usepackage[colorlinks=true,linkcolor=blue,citecolor=blue,urlcolor=blue]{hyperref}

\newtheorem{theorem}{Theorem}[section]
\newtheorem{proposition}[theorem]{Proposition}
\newtheorem{lemma}[theorem]{Lemma}
\theoremstyle{remark}
\newtheorem{remark}[theorem]{Remark}

\newcommand{\Ric}{\operatorname{Ric}}
\newcommand{\tr}{\operatorname{tr}}

\title[Counterexample to a sixth-order maximum principle]
{A counterexample to a strong maximum principle for the sixth-order GJMS operator}

\author{Liuwei Gong}
\address{Liuwei Gong\newline
Department of Mathematics, The Chinese University of Hong Kong,
Shatin, Hong Kong}
\email{lwgong@math.cuhk.edu.hk}

\author{Mingxiang Li}
\address{Mingxiang Li\newline
Department of Mathematics, The Chinese University of Hong Kong,
Shatin, Hong Kong}
\email{limx@smail.nju.edu.cn}

\author{Juncheng Wei}
\address{Juncheng Wei\newline
Department of Mathematics, The Chinese University of Hong Kong,
Shatin, Hong Kong}
\email{wei@math.cuhk.edu.hk}

\date{}

\begin{document}
\begin{abstract}
We exhibit an explicit closed seven-dimensional Riemannian manifold
\[
(M,g)=\mathbb S^2(1)\times \mathbb S^5\left(\frac1{100}\right),
\]
where the displayed parameters denote sectional curvatures, for which
\(\Ric_g>0\), and hence \(Q_g^{(2)}>0\).  Moreover,
\[
Q^{(4)}_g>0,\qquad Q^{(6)}_g>0,
\]
and the sixth-order GJMS operator \(P_{6,g}\) is strictly positive as a
self-adjoint operator, but nevertheless \(P_{6,g}\) fails the strong maximum
principle.  The failure is caused by a nonconstant positive eigenvalue of
\(P_{6,g}\) lying strictly below the eigenvalue of the constant mode.
The example also has \(Y_2(M,[g])>0\) and \(Y_4(M,[g])>0\), while
\(P_{6,g}\) does not have a positive Green function.  It disproves
Conjecture~1 of Andrade, Piccione, and Wei and its general-order
formulation by Case and Gover.
\end{abstract}

\maketitle

\section{Introduction}

In the study of the \(Q\)-curvature \(Q^{(2k)}_g\) of order \(2k\), the
strong maximum principle for the GJMS operator \(P_{2k,g}\) is a central
topic and is frequently imposed in a priori estimates.  Let \((M^n,g)\)
be a compact smooth Riemannian manifold of dimension \(n\ge 2k+1\).
We say that \(P_{2k,g}\) satisfies the strong maximum principle,
abbreviated (SMP), if every smooth function \(u\) on \(M\) satisfies
\begin{equation}\label{eq:SMP}
P_{2k,g}u\ge0
\quad\Longrightarrow\quad
u>0\ \text{or}\ u\equiv0.
\end{equation}
With our normalization,
\[
P_{2k,g}(1)=\frac{n-2k}{2}Q^{(2k)}_g.
\]

When \(k=1\), \(P_{2,g}\) is, up to a positive constant factor, the
conformal Laplacian
\[
L_gu=-\frac{4(n-1)}{n-2}\Delta_g u+R_g u,
\]
where \(R_g\) is the scalar curvature.  With the normalization above,
\[
Q^{(2)}_g=\frac{R_g}{2(n-1)}.
\]
For \(n\ge2k+1\), define the Yamabe invariant of order \(2k\) by
\[
Y_{2k}(M,[g]):=\inf_{u\in C^\infty(M)\setminus\{0\}}
\frac{\int_M uP_{2k,g}u\,dv_g}
{\left(\int_M |u|^{\frac{2n}{n-2k}}\,dv_g\right)^{\frac{n-2k}{n}}}.
\]
The Yamabe-type problem for \(Q^{(2k)}_g\) asks for a positive solution of
\[
P_{2k,g}u=Cu^{\frac{n+2k}{n-2k}}
\]
for some constant \(C\).  In the standard variational approach, the
strong maximum principle is crucial for obtaining a positive minimizer.
Further discussion can be found in \cite{MV}.

For \(k=1\), it is well known that
\begin{equation}\label{R_g-positive}
Q^{(2)}_g>0
\quad\text{or}\quad
Y_2(M,[g])>0
\end{equation}
implies that \(P_{2,g}\) satisfies (SMP).

For \(k\ge2\), the operator \(P_{2k,g}\) on a positive Einstein
manifold factors into second-order operators \cite{Gover,FG}; in
particular, it satisfies (SMP).  The situation on a general compact
manifold is substantially different.

For \(k=2\), \(P_{4,g}\) is the Paneitz operator.  Gursky and
Malchiodi \cite{GM} found a natural geometric condition analogous to
\eqref{R_g-positive}: if \(n\ge5\), then
\begin{equation}\label{GM-condition}
Q^{(2)}_g>0
\quad\text{and}\quad
Q^{(4)}_g>0
\quad\Longrightarrow\quad
P_{4,g}\text{ satisfies (SMP)}.
\end{equation}
Hang and Yang \cite{HY-IMRN,HY-CPAM} generalized this to
\[
Y_2(M,[g])>0
\quad\text{and}\quad
Q^{(4)}_g>0
\quad\Longrightarrow\quad
P_{4,g}\text{ satisfies (SMP)}.
\]
For \(n\ge6\), Gursky, Hang, and Lin \cite{GHL} showed that
\[
Y_2(M,[g])>0
\quad\text{and}\quad
Y_4(M,[g])>0
\quad\Longrightarrow\quad
P_{4,g}\text{ satisfies (SMP)}.
\]
The same statement for \(n=5\) was established in \cite{Li}.  See
\cite{HY-Lecture} for an overview and
\cite{GKW,GLW,HR,LX,Ligang,MP} for related fourth-order
\(Q\)-curvature problems.

For \(k\ge3\), much less is known.  Case and Malchiodi \cite{CM}
considered Einstein manifolds \((N^\ell,g)\) and \((M^{d-\ell},h)\)
satisfying
\[
\Ric_g=-(\ell-1)\lambda g,
\qquad
\Ric_h=(d-\ell-1)\lambda h,
\qquad
\lambda>0.
\]
On the product \((N^\ell\times M^{d-\ell},g\oplus h)\), they proved
that \(P_{2k,g\oplus h}\) satisfies (SMP) whenever
\(d\ge D(k,\ell)\) for a suitable constant \(D(k,\ell)>0\).

In view of \eqref{R_g-positive} and \eqref{GM-condition}, it is natural
to ask whether, for \(k\ge3\),
\begin{equation}\label{Q_2kpositive}
Q^{(2i)}_g>0\quad\text{for }1\le i\le k
\quad\Longrightarrow\quad
P_{2k,g}\text{ satisfies (SMP)}.
\end{equation}
In a related direction, Andrade, Piccione, and Wei
\cite[Conjecture~1]{APW} conjectured that, on every closed Riemannian
manifold of dimension \(n\ge7\), the conditions
\[
Y_2(M,[g])\ge0,
\qquad
Y_4(M,[g])\ge0,
\qquad
Q^{(6)}_g\ge0,
\qquad
Q^{(6)}_g\not\equiv0
\]
imply that \(P_{6,g}\) has a positive Green function.  They observed
that the standard product \(\mathbb S^1\times\mathbb S^{n-1}\)
satisfies these hypotheses.  Motivated by this conjecture, Case and
Gover \cite[Conjecture~4.4]{CG} formulated the corresponding
general-order conjecture: if \(Y_{2j}(M,[g])>0\) for
\(1\le j\le k-1\) and some \(\widehat g\in[g]\) satisfies
\(Q^{(2k)}_{\widehat g}\ge0\) and
\(Q^{(2k)}_{\widehat g}\not\equiv0\), then \(P_{2k,g}\) satisfies (SMP).

Related work asks when positivity of a top-order \(Q\)-curvature forces
positivity of the lower-order curvatures.  At fourth order, Fazly, Wei,
and Xu \cite[Theorem~1.3]{FWX} obtained a pointwise estimate for a class
of bounded positive solutions of
\(\Delta^2u=|x|^a u^p\) on \(\mathbb R^n\); geometrically, it implies
that the conformal metric \(u^{4/(n-4)}|dx|^2\) has positive scalar
curvature.  For conformally Euclidean metrics on \(\mathbb R^n\) with
\(n>2m\), Li and Xu \cite{LiXu} proved, under their slow-decay barrier hypothesis,
that positive \(Q^{(2m)}\) implies positivity of \(Q^{(2)}\) and
\(Q^{(4)}\).
V\'etois and Zeitler \cite{VZ} showed that, when \(m\ge4\) and the
dimension is sufficiently large, even positive \(Q^{(2m)}\) need not
force nonnegative \(Q^{(6)}\). This unexpected finding reveals that higher-order $Q$-curvatures $Q_g^{(2k)}$ may contain more diverse and intricate structures than previously anticipated.

In this paper, we give an explicit counterexample to
\eqref{Q_2kpositive}, to \cite[Conjecture~1]{APW}, and, with \(k=3\),
to its general-order formulation in \cite[Conjecture~4.4]{CG}.
Although maximum principles fail for general higher-order elliptic
operators \cite{CH}, to the best of our knowledge this is the first
explicit example satisfying the hypotheses of Andrade--Piccione--Wei---in
fact with \(\Ric>0\), \(Q^{(4)}>0\), \(Q^{(6)}>0\), and \(P_6\) strictly
positive as a quadratic form---for which \(P_6\) fails positivity
preservation.  The product metric is not locally conformally flat, so the
example does not address variants of these statements with local
conformal flatness imposed.

\pagebreak[3]
\begin{theorem}\label{thm:counterexample}
Let
\[
(M^7,g)=\mathbb S^2(1)\times \mathbb S^5\left(\frac1{100}\right),
\]
where the displayed parameters denote sectional curvatures.  Then
\(\Ric_g>0\), and hence \(Q_g^{(2)}>0\).  Moreover,
\[
Q^{(4)}_g>0,
\qquad
Q^{(6)}_g>0,
\]
and \(P_{6,g}\) is strictly positive in the sense that
\[
\int_M vP_{6,g}v\,dv_g>0
\qquad
\text{for every }v\not\equiv0.
\]
Nevertheless, there is a smooth function \(u\) such that
\[
P_{6,g}u>0
\qquad\text{and}\qquad
\min_Mu<0.
\]
In particular, \(P_{6,g}\) does not satisfy the strong maximum principle.
Moreover,
\[
Y_2(M,[g])>0,
\qquad
Y_4(M,[g])>0,
\]
but \(P_{6,g}\) does not have a positive Green function.  Thus this metric
gives a counterexample to \cite[Conjecture~1]{APW} and, with \(k=3\), to
its general-order formulation in \cite[Conjecture~4.4]{CG}.
\end{theorem}

The counterexample rests on the following elementary spectral mechanism.

\begin{lemma}\label{lem:spectral-obstruction}
Let \(M^n\) be closed, let \(n>6\), and suppose that
\(Q^{(6)}_g\equiv q\) for some constant \(q>0\).  If \(P_{6,g}\) has a
nonconstant eigenfunction
\(\varphi\) with eigenvalue \(\mu\) satisfying
\[
0<\mu<\frac{n-6}{2}q,
\]
then \(P_{6,g}\) does not satisfy (SMP).
\end{lemma}

\begin{proof}
The constant function is an eigenfunction with eigenvalue
\((n-6)q/2\).  Since \(P_{6,g}\) is self-adjoint and
\(\mu\ne(n-6)q/2\), the eigenfunctions \(\varphi\) and \(1\) are
orthogonal.  Thus
\[
\int_M\varphi\,dv_g=0,
\]
so \(\varphi\) changes sign.  Normalize it by
\(\min_M\varphi=-1\), choose
\[
\frac{2\mu}{(n-6)q}<c<1,
\]
and set \(u=c+\varphi\).  Then \(\min_Mu=c-1<0\), whereas
\[
P_{6,g}u
=c\frac{n-6}{2}q+\mu\varphi
\ge c\frac{n-6}{2}q-\mu>0.
\]
Hence \eqref{eq:SMP} fails.
\end{proof}

We apply the lemma to
\(\mathbb S^2(1)\times\mathbb S^5(r)\), carrying out the curvature and
spectral calculations for general \(r>0\) before choosing
\(r=1/100\).

\begin{remark}
Propositions \ref{prop:product-family} and \ref{prop:P6-positive} below
give a family of metrics exhibiting the spectral obstruction in Lemma
\ref{lem:spectral-obstruction} while the full operator remains strictly
positive.  The value \(r=1/100\) is a convenient rational representative.
This contrasts with the fourth-order theory: the positivity assumptions
of Gursky--Malchiodi imply the strong maximum principle for the Paneitz
operator \cite{GM}, whereas positivity of \(Q^{(6)}_g\), even together
with strict positivity of \(P_{6,g}\), does not suffice.
\end{remark}

\begin{remark}
The spectral mechanism is not specific to order six.  If \(n>2k\),
\(Q^{(2k)}_g\equiv q>0\), and \(P_{2k,g}\) has a nonconstant
eigenfunction with eigenvalue
\[
0<\mu<\frac{n-2k}{2}q,
\]
then the proof of Lemma \ref{lem:spectral-obstruction} applies
verbatim.  It is therefore natural to expect this mechanism to help
produce counterexamples to the strong maximum principle for general
\(P_{2k}\).
\end{remark}

\begin{remark}
Within the present family, an SMP does hold at the larger parameter
\(r=1/4\).  Indeed, the product metric is then Einstein, with
\(\Ric_g=g\), and the Einstein factorization of Case--Malchiodi
\cite[(1.2)]{CM} gives
\[
P_{6,g}
=\left(-\Delta_g+\frac{35}{24}\right)
 \left(-\Delta_g+\frac98\right)
 \left(-\Delta_g+\frac{11}{24}\right).
\]
The strong maximum principle for each second-order factor therefore
implies the strong maximum principle for \(P_{6,g}\).  It would be
interesting to determine where the transition between the two regimes
occurs.
\end{remark}

\section{Geometric formulas and the product family}

We use the sign convention \(-\Delta_g\ge0\).  The Ricci tensor,
scalar curvature, Schouten tensor, and normalized scalar curvature are
denoted by \(\Ric_g\), \(R_g\), \(A_g\), and \(J_g\), respectively, with
\[
A_g
=\frac1{n-2}
\left(\Ric_g-\frac{R_g}{2(n-1)}g\right),
\qquad
J_g:=\sigma_1(A_g)=\frac{R_g}{2(n-1)}.
\]
We use the Cotton and Bach tensor conventions
\[
(C_g)_{ijk}:=\nabla_k(A_g)_{ij}-\nabla_j(A_g)_{ik},
\qquad
(B_g)_{ij}:=\nabla^k(C_g)_{ijk}-(A_g)^{kl}(W_g)_{kijl}.
\]
Here \(W_g\) denotes the Weyl tensor.
We similarly write \(P_{2k,g}\) and \(Q^{(2k)}_g\) for the GJMS operator
and \(Q\)-curvature associated with \(g\).
The Riemannian volume measure is denoted by \(dv_g\).
Whenever the metric is clear, we omit the subscript \(g\) from all
metric-dependent quantities and operators.

General explicit formulas for GJMS operators and their \(Q\)-curvatures
were obtained by Juhl \cite{Juhl}.  We use the normalizations of
\(P_{6,g}\) and \(Q^{(6)}_g\) in
Chen--Hou \cite[(1-1)--(1-2)]{CH}.  In particular,
\begin{equation}\label{eq:P6-general}
-P_{6,g}
=
\Delta_g^3+\Delta_g\delta_gT_{2,g}d
+\delta_gT_{2,g}d\,\Delta_g
+\frac{n-2}{2}\Delta_g(J_g\Delta_g)
+\delta_gT_{4,g}d-\frac{n-6}{2}Q^{(6)}_g,
\end{equation}
where
\begin{align}
T_{2,g}&=(n-2)J_g\,g-8A_g,\label{eq:T2}\\
T_{4,g}&=
-\frac{3n^2-12n-4}{4}J_g^2g
+4(n-4)|A_g|_g^2g
+8(n-2)J_gA_g
+(n-6)(\Delta_gJ_g)g \notag\\
&\qquad
-48A_g^2-\frac{16}{n-4}B_g. \label{eq:T4}
\end{align}
Here \(d\) denotes exterior differentiation, \(\delta_g\) is its
\(L^2\)-formal adjoint, and a symmetric two-tensor between \(\delta_g\)
and \(d\) acts on one-forms by contraction.  We regard \(A_g\) as an
endomorphism using \(g\), and \(A_g^2\) denotes its square.
We also use
\begin{equation}\label{eq:Q4}
Q^{(4)}_g=-\Delta_gJ_g-2|A_g|_g^2+\frac n2J_g^2.
\end{equation}

We now specialize these formulas to a product family.  Throughout,
\(\mathbb S^m(\kappa)\) denotes the round \(m\)-sphere of sectional
curvature \(\kappa>0\).  Let \((\mathbb S^2,g_1)\) have sectional
curvature \(1\), and let \((\mathbb S^5,g_2)\) have sectional curvature
\(r>0\).  We write
\[
(M^7,g)=\mathbb S^2(1)\times\mathbb S^5(r),
\qquad
g=g_1\oplus g_2.
\]
Then
\[
\Ric
=
g_1\oplus4r\,g_2,
\]
and hence \(R=2+20r\), and
\begin{equation}\label{eq:J}
J=\frac{R}{12}=\frac{1+10r}{6}.
\end{equation}

The Schouten tensor is parallel and has two eigenvalues:
\begin{equation}\label{eq:A-eigen}
A
=
\frac{1}{6}(1-2r)\,g_1\oplus \frac{1}{30}(14r-1)\,g_2.
\end{equation}
For this calculation, set
\[
a=\frac{1-2r}{6},
\qquad
b=\frac{14r-1}{30}.
\]
If \(i,j\) are tangent to \(\mathbb S^2\) and \(\alpha,\beta\) are
tangent to \(\mathbb S^5\), then the mixed Weyl component is
\[
W_{\alpha i j\beta}
=(a+b)(g_1)_{ij}(g_2)_{\alpha\beta}.
\]
Since \(a+b=2(1+r)/15>0\), the product metric is not locally
conformally flat.
The Cotton tensor vanishes, so \(B_{ij}=-A^{kl}W_{kijl}\), and therefore
\[
B|_{T\mathbb S^2}=5(a-b)(a+b)g_1,
\qquad
B|_{T\mathbb S^5}=2(b-a)(a+b)g_2.
\]
Equivalently,
\begin{equation}\label{eq:Bach}
B
=
\left[-\frac{2}{15}(1+r)(4r-1)\right]g_1
\oplus
\left[\frac{4}{75}(1+r)(4r-1)\right]g_2.
\end{equation}
Its two coefficients have weighted trace zero, as required by
\(\tr_gB=0\).

Since all the curvature quantities above are parallel, the curvature
formulas reduce to algebra.  For example,
\[
|A|^2
=
2\left(\frac{1-2r}{6}\right)^2
+5\left(\frac{14r-1}{30}\right)^2.
\]
For \(n=7\) and parallel \(A\), the Chen--Hou formula reduces to
\[
Q^{(6)}_g
=48\sigma_3(A)+\frac{16}{3}\langle B,A\rangle
-4J|A|^2+\frac{13}{4}J^3.
\]
Here the two eigenvalues in \eqref{eq:A-eigen}, with multiplicities
two and five, give
\[
\sigma_3(A)=5a^2b+20ab^2+10b^3,
\qquad
\langle B,A\rangle=10(a+b)(a-b)^2.
\]
Substitution into this identity and \eqref{eq:Q4} gives
\begin{align}
Q^{(4)}
=&\frac{1}{40}(284r^2+108r-1),\label{eq:Q4-r}\\
Q^{(6)}
=&\frac{1}{1440}(3448r^3+42404r^2-3286r+283).\notag
\end{align}

Substituting \eqref{eq:J}, \eqref{eq:A-eigen}, and \eqref{eq:Bach}
into \eqref{eq:T2}--\eqref{eq:T4} gives
\[
T_2
=
\frac{22r-1}{2}\,g_1
\oplus
\frac{46r+11}{10}\,g_2
\]
and
\[
T_4
=
\left(-\frac{35964r^2-2612r+439}{720}\right)g_1
\oplus
\left(-\frac{20588r^2+40316r-1197}{3600}\right)g_2.
\]

\section{Spectral analysis and the explicit obstruction}

We first describe the admissible parameters in the full product family
and then treat \(r=1/100\) as an explicit application.

\begin{proposition}\label{prop:product-family}
Let \(r_0=(20\sqrt2-27)/142\), and let \(r^*>0\) be the unique
positive root of \(1197-72416r-201188r^2=0\).  For every
\(r\in(r_0,r^*)\), the product
\(\mathbb S^2(1)\times\mathbb S^5(r)\) satisfies
\[
\Ric_g>0,
\qquad
Q^{(4)}_g>0,
\qquad
Q^{(6)}_g>0.
\]
Moreover, \(P_{6,g}\) has a nonconstant eigenvalue in
\((0,Q^{(6)}_g/2)\), and hence does not satisfy the strong maximum
principle.
\end{proposition}

\begin{proof}
Let \(\phi_\ell\) and \(\psi_m\) be spherical harmonics of degrees
\(\ell\) and \(m\) on the two factors.  Their Laplace eigenvalues are
\(x=\ell(\ell+1)\) and \(y=r\,m(m+4)\), respectively.  Because
\(J,T_2,T_4\) are parallel, direct substitution into
\eqref{eq:P6-general} gives
\[
P_{6,g}(\phi_\ell\psi_m)
=\Lambda_r(x,y)\phi_\ell\psi_m,
\]
where
\begin{equation}\label{eq:Lambda-r}
\begin{aligned}
\Lambda_r(x,y)
={}&(x+y)^3
+\frac{214r-17}{12}x^2
+\frac{686r+11}{30}xy
+\frac{302r+107}{60}y^2\\
&+\frac{35964r^2-2612r+439}{720}x\\
&+\frac{20588r^2+40316r-1197}{3600}y\\
&+\frac{3448r^3+42404r^2-3286r+283}{2880}.
\end{aligned}
\end{equation}

The constant mode has eigenvalue
\[
\Lambda_r(0,0)
=\frac12Q^{(6)}
=
\frac{3448r^3+42404r^2-3286r+283}{2880},
\]
where we used \(n=7\) and \(P_6(1)=Q^{(6)}/2\).  For a first spherical
harmonic \(\psi\) on the second factor,
\[
-\Delta_{g_2}\psi=5r\,\psi,
\]
and hence
\[
P_6\psi=\Lambda_r(0,5r)\psi,
\qquad
\Lambda_r(0,5r)
=
\frac{808200r^3+332068r^2-8074r+283}{2880}.
\]
Consequently,
\begin{equation}\label{eq:gap-r}
\Lambda_r(0,0)-\Lambda_r(0,5r)
=
\frac{r}{720}
\left(1197-72416r-201188r^2\right).
\end{equation}

For \(r>0\), \eqref{eq:Q4-r} is positive precisely when
\[
r>r_0
=\frac{20\sqrt2-27}{142}
\approx0.00904416,
\]
whereas the right-hand side of \eqref{eq:gap-r} is positive precisely
when \(0<r<r^*\), with \(r^*\approx0.0158330\).  The polynomial defining
\(r^*\) is strictly decreasing for \(r>0\), and
\[
1197-72416\left(\frac2{125}\right)
-201188\left(\frac2{125}\right)^2
=-\frac{205627}{15625}<0.
\]
Hence \(r^*<2/125\), and for \(0<r<r^*\),
\[
2880\Lambda_r(0,0)>283-3286\left(\frac2{125}\right)>0,
\qquad
2880\Lambda_r(0,5r)>283-8074\left(\frac2{125}\right)>0.
\]
Together with \(\Ric=g_1\oplus4r\,g_2>0\), the preceding estimates prove
all the asserted inequalities.  Lemma \ref{lem:spectral-obstruction}
then applies with
\(q=Q^{(6)}\) and \(\mu=\Lambda_r(0,5r)\).
\end{proof}

We now apply Proposition \ref{prop:product-family} at \(r=1/100\).
This value lies in \((r_0,r^*)\), since
\[
284r^2+108r-1=\frac{271}{2500}>0,
\qquad
1197-72416r-201188r^2=\frac{1131803}{2500}>0
\]
at this value.  Substitution into the formulas of Section 2 gives
\(\Ric_g=g_1\oplus g_2/25>0\),
\(Q^{(4)}_g=271/100000>0\), and
\(Q^{(6)}_g=3533109/20000000>0\).  For a first spherical harmonic on
the second factor, \(m=1\), so at \(r=1/100\) its Laplace eigenvalue is
\(y=rm(m+4)=(1/100)\cdot1\cdot5=1/20\).  Therefore, for the constant
mode and this first mode,
\begin{equation}\label{eq:spectral-gap}
\begin{gathered}
\Lambda_{1/100}(0,0)
=
\frac{3533109}{40000000},
\qquad
\Lambda_{1/100}\left(0,\frac1{20}\right)
=
\frac{9451}{115200},\\
\Lambda_{1/100}(0,0)
-\Lambda_{1/100}\left(0,\frac1{20}\right)
=
\frac{1131803}{180000000}>0,
\qquad
\frac{\Lambda_{1/100}\left(0,\frac1{20}\right)}
{\Lambda_{1/100}(0,0)}
=
\frac{29534375}{31797981}
<\frac{19}{20}.
\end{gathered}
\end{equation}

The last inequality explains the choice of \(19/20\) below.  Choose a
first spherical harmonic \(\psi\) on the second factor, normalize it by
\(\min\psi=-1\), and note that
\(P_6\psi=\Lambda_{1/100}(0,1/20)\psi\).  Set
\begin{equation}\label{eq:u}
u=\frac{19}{20}+\psi,
\qquad
\min_Mu=\frac{19}{20}-1=-\frac1{20}<0.
\end{equation}
Since \(\psi\ge-1\), \eqref{eq:spectral-gap} gives
\begin{equation}\label{eq:P6u-positive}
P_6u
\ge
\frac{19}{20}\Lambda_{1/100}(0,0)
-\Lambda_{1/100}\left(0,\frac1{20}\right)
=
\frac{13474139}{7200000000}>0.
\end{equation}

\section{Positivity of $P_{6,g}$ and proof of the theorem}

The positivity argument also admits a general-parameter formulation.

\begin{proposition}\label{prop:P6-positive}
For every \(r\in(r_0,r^*)\), the operator \(P_{6,g}\) on
\(\mathbb S^2(1)\times\mathbb S^5(r)\) is strictly positive:
\[
\int_M vP_{6,g}v\,dv_g>0
\qquad
\text{for every }v\not\equiv0.
\]
\end{proposition}

\begin{proof}
Choose orthonormal bases of spherical harmonics on the two factors.
Their products form an orthonormal eigenbasis, so by \eqref{eq:Lambda-r}
it suffices to prove
\(\Lambda_r(x,y)>0\) for
\[
x=\ell(\ell+1),
\qquad
y=r\,m(m+4),
\qquad
\ell,m\ge0.
\]
The polynomial \(284r^2+108r-1\) is strictly increasing for \(r>0\), and
\[
284\left(\frac1{125}\right)^2
+108\left(\frac1{125}\right)-1
=-\frac{1841}{15625}<0.
\]
Since \(r_0\) is its unique positive root, while the proof of Proposition
\ref{prop:product-family} gives \(r^*<2/125\), we have
\[
\frac1{125}<r_0<r<r^*<\frac2{125}.
\]

If \(\ell\ge1\), then \(x\ge2\), and differentiation of
\eqref{eq:Lambda-r} gives
\[
\begin{aligned}
\partial_x\Lambda_r(x,y)
={}&x\left(3x+\frac{214r-17}{6}\right)+6xy+3y^2\\
&+\frac{686r+11}{30}y
+\frac{35964r^2-2612r+439}{720}>0.
\end{aligned}
\]
Indeed, the first term is positive for \(x\ge2\), and the last
numerator is larger than \(439-2(2612)/125>0\).  Moreover,
\[
\partial_y\Lambda_r(2,y)
=3y^2+\frac{302r+467}{30}y
+\frac{20588r^2+204956r+44643}{3600}>0.
\]
It follows that
\[
\Lambda_r(x,y)
\ge\Lambda_r(2,y)
\ge\Lambda_r(2,0)
=\frac{431}{360}r^3+\frac{82529}{720}r^2
+\frac{90629}{1440}r+\frac{701}{192}>0.
\]

It remains to take \(\ell=0\).  The cases \(m=0,1\) are positive by
Proposition \ref{prop:product-family}.  If \(m\ge2\), then \(y\ge12r\),
and
\[
\partial_y^2\Lambda_r(0,y)
=6y+\frac{302r+107}{30}>0.
\]
At the left endpoint,
\[
\begin{aligned}
\partial_y\Lambda_r(0,12r)
&=\frac{502667}{900}r^2+\frac{48599}{900}r-\frac{133}{400}\\
&>\frac{48599}{900}\frac1{125}-\frac{133}{400}
=\frac{44771}{450000}>0.
\end{aligned}
\]
Thus \(y\mapsto\Lambda_r(0,y)\) is increasing for \(y\ge12r\).
Finally,
\[
\begin{aligned}
\Lambda_r(0,12r)
={}&\frac{4540723}{1800}r^3+\frac{1461277}{3600}r^2
-\frac{36943}{7200}r+\frac{283}{2880}\\
>{}&\frac{283}{2880}
-\frac{36943}{7200}\frac2{125}
=\frac{9701}{600000}>0.
\end{aligned}
\]
Hence every eigenvalue is positive.  Writing
\(v=\sum_{\ell,m}v_{\ell m}\phi_\ell\psi_m\) for the expansion of \(v\)
in the orthonormal product basis, with multiplicities understood, we have
\[
\int_M vP_{6,g}v\,dv_g
=\sum_{\ell,m}
\Lambda_r\bigl(\ell(\ell+1),r m(m+4)\bigr)
|v_{\ell m}|^2>0
\]
for every \(v\not\equiv0\), as claimed.
\end{proof}

For the application \(r=1/100\), the key endpoint computations become
\[
\begin{aligned}
\Lambda_{1/100}(2,0)
&=\frac{1545074381}{360000000}>0,\\
\partial_y\Lambda_{1/100}\left(0,\frac3{25}\right)
&=\frac{2370067}{9000000}>0,\\
\Lambda_{1/100}\left(0,\frac3{25}\right)
&=\frac{54040691}{600000000}>0.
\end{aligned}
\]
Together with the modes \(m=0,1\) in \eqref{eq:spectral-gap},
Proposition \ref{prop:P6-positive} gives
\[
\int_M vP_{6,g}v\,dv_g>0
\qquad(v\not\equiv0).
\]

\begin{proof}[{\bf Proof of Theorem \ref{thm:counterexample}}]
The curvature inequalities were verified in the \(r=1/100\)
calculation in Section 3.  The function \eqref{eq:u} violates the strong
maximum principle by \eqref{eq:P6u-positive}, and Proposition
\ref{prop:P6-positive} gives strict positivity of the quadratic form.

Since \(\Ric_g>0\), we have \(R_g>0\), and hence
\(Y_2(M,[g])>0\).  Since also \(Q^{(4)}_g>0\), the positivity theorem of
Gursky--Malchiodi \cite[Proposition~2.3]{GM} implies that \(P_{4,g}\) is
positive; elliptic coercivity and the Sobolev embedding then give
\(Y_4(M,[g])>0\).

Finally, strict positivity makes \(P_{6,g}\) invertible.  If its Green
function \(G_6(x,y)\) were positive, then, with \(f=P_{6,g}u>0\), the
Green representation formula would give
\[
u(x)=\int_M G_6(x,y)f(y)\,dv_g(y)>0,
\]
contrary to \(\min_Mu<0\).  Thus the positive-Green-function property
conjectured in \cite{APW} fails.  Together with the already established
positivity of \(Y_2(M,[g])\) and \(Y_4(M,[g])\), the failure of (SMP) also
contradicts \cite[Conjecture~4.4]{CG} with \(k=3\).
\end{proof}

\bigskip
\noindent\textbf{Acknowledgments.}
The authors gratefully acknowledge support from the Research Grants
Council of Hong Kong (RGC) through the project:
\emph{On Fujita equation in the critical or supercritical regime}.

\medskip
\noindent\textbf{AI disclosure.}
Generative artificial intelligence tools were used to assist with
language editing, notation standardization, structural reorganization,
and checking algebraic computations.  The authors reviewed the resulting
text and take full responsibility for its mathematical content.

\begingroup
\hbadness=10000
\bibliographystyle{abbrv}
\bibliography{ref}

@misc{APW,
  author        = {Andrade, Jo{\~a}o Henrique and Piccione, Paolo and Wei, Juncheng},
  title         = {Nonuniqueness results for constant sixth order {$Q$}-curvature
                   metrics on spheres with higher dimensional singularities},
  howpublished  = {arXiv:2306.00679},
  year          = {2023},
  eprint        = {2306.00679},
  archivePrefix = {arXiv},
  primaryClass  = {math.DG},
  doi           = {10.48550/arXiv.2306.00679},
  url           = {https://arxiv.org/abs/2306.00679},
}

@article{CH,
  author  = {Chen, Xuezhang and Hou, Fei},
  title   = {Remarks on {GJMS} operator of order six},
  journal = {Pacific J. Math.},
  volume  = {289},
  number  = {1},
  pages   = {35--70},
  year    = {2017},
  doi     = {10.2140/pjm.2017.289.35},
}

@article{GM,
  author  = {Gursky, Matthew J. and Malchiodi, Andrea},
  title   = {A strong maximum principle for the {Paneitz} operator and a
             non-local flow for the {$Q$}-curvature},
  journal = {J. Eur. Math. Soc.},
  volume  = {17},
  number  = {9},
  pages   = {2137--2173},
  year    = {2015},
  doi     = {10.4171/JEMS/553},
}

@article{CM,
  author  = {Case, Jeffrey S. and Malchiodi, Andrea},
  title   = {A factorization of the {GJMS} operators of special {Einstein}
             products and applications},
  journal = {J. Lond. Math. Soc. (2)},
  volume  = {110},
  number  = {5},
  pages   = {Paper No. e70023, 17},
  year    = {2024},
  doi     = {10.1112/jlms.70023},
  url     = {https://doi.org/10.1112/jlms.70023},
}

@article {Gover,
    AUTHOR = {Gover, A. R.},
     TITLE = {Laplacian operators and {$Q$}-curvature on conformally
              {E}instein manifolds},
   JOURNAL = {Math. Ann.},
  FJOURNAL = {Mathematische Annalen},
    VOLUME = {336},
      YEAR = {2006},
    NUMBER = {2},
     PAGES = {311--334},
      ISSN = {0025-5831,1432-1807},
   MRCLASS = {58J60 (53C20 53C21)},
  MRNUMBER = {2244375},
MRREVIEWER = {Mohameden\ Ahmedou},
       DOI = {10.1007/s00208-006-0004-z},
       URL = {https://doi.org/10.1007/s00208-006-0004-z},
}

@book {FG,
    AUTHOR = {Fefferman, Charles and Graham, C. Robin},
     TITLE = {The ambient metric},
    SERIES = {Annals of Mathematics Studies},
    VOLUME = {178},
 PUBLISHER = {Princeton University Press, Princeton, NJ},
      YEAR = {2012},
     PAGES = {x+113},
      ISBN = {978-0-691-15313-1},
   MRCLASS = {53A30 (53A55 53C20)},
  MRNUMBER = {2858236},
MRREVIEWER = {Michael\ G.\ Eastwood},
}

@article {HY-CPAM,
    AUTHOR = {Hang, Fengbo and Yang, Paul C.},
     TITLE = {{$Q$}-curvature on a class of manifolds with dimension at
              least 5},
   JOURNAL = {Comm. Pure Appl. Math.},
  FJOURNAL = {Communications on Pure and Applied Mathematics},
    VOLUME = {69},
      YEAR = {2016},
    NUMBER = {8},
     PAGES = {1452--1491},
      ISSN = {0010-3640,1097-0312},
   MRCLASS = {53A30 (58J05)},
  MRNUMBER = {3518237},
MRREVIEWER = {Yuxin\ Ge},
       DOI = {10.1002/cpa.21623},
       URL = {https://doi.org/10.1002/cpa.21623},
}

@article {HY-IMRN,
    AUTHOR = {Hang, Fengbo and Yang, Paul C.},
     TITLE = {Sign of {G}reen's function of {P}aneitz operators and the
              {$Q$} curvature},
   JOURNAL = {Int. Math. Res. Not. IMRN},
  FJOURNAL = {International Mathematics Research Notices. IMRN},
    VOLUME = {2015},
      YEAR = {2015},
    NUMBER = {19},
     PAGES = {9775--9791},
      ISSN = {1073-7928,1687-0247},
   MRCLASS = {53A30 (35J08 53C20)},
  MRNUMBER = {3431611},
MRREVIEWER = {Ali\ Maalaoui},
       DOI = {10.1093/imrn/rnu247},
       URL = {https://doi.org/10.1093/imrn/rnu247},
}

@incollection {HY-Lecture,
    AUTHOR = {Hang, Fengbo and Yang, Paul C.},
     TITLE = {Lectures on the fourth-order {$Q$} curvature equation},
    EDITOR = {Han, Fei and Xu, Xingwang and Zhang, Weiping},
 BOOKTITLE = {Geometric analysis around scalar curvatures},
    SERIES = {Lect. Notes Ser. Inst. Math. Sci. Natl. Univ. Singap.},
    VOLUME = {31},
     PAGES = {1--33},
 PUBLISHER = {World Sci. Publ., Hackensack, NJ},
      YEAR = {2016},
      ISBN = {978-981-3100-54-1},
       DOI = {10.1142/9789813100558_0001},
       URL = {https://doi.org/10.1142/9789813100558_0001},
   MRCLASS = {53A30 (35J61 35R01 58J05)},
  MRNUMBER = {3618119},
MRREVIEWER = {Yongbing\ Zhang},
}

@article {MV,
    AUTHOR = {Mazumdar, Saikat and V\'etois, J\'er\^ome},
     TITLE = {Existence results for the higher-order {$Q$}-curvature
              equation},
   JOURNAL = {Calc. Var. Partial Differential Equations},
  FJOURNAL = {Calculus of Variations and Partial Differential Equations},
    VOLUME = {63},
      YEAR = {2024},
    NUMBER = {6},
     PAGES = {Paper No. 151, 29},
      ISSN = {0944-2669,1432-0835},
   MRCLASS = {53C18 (35A01 35G20 35J30 35J35 35R01)},
  MRNUMBER = {4765816},
MRREVIEWER = {Zhongyuan\ Liu},
       DOI = {10.1007/s00526-024-02757-x},
       URL = {https://doi.org/10.1007/s00526-024-02757-x},
}

@article {GHL,
    AUTHOR = {Gursky, Matthew J. and Hang, Fengbo and Lin, Yueh-Ju},
     TITLE = {Riemannian manifolds with positive {Y}amabe invariant and
              {P}aneitz operator},
   JOURNAL = {Int. Math. Res. Not. IMRN},
  FJOURNAL = {International Mathematics Research Notices. IMRN},
    VOLUME = {2016},
      YEAR = {2016},
    NUMBER = {5},
     PAGES = {1348--1367},
      ISSN = {1073-7928,1687-0247},
   MRCLASS = {53A30 (58J05)},
  MRNUMBER = {3509928},
MRREVIEWER = {Ali\ Maalaoui},
       DOI = {10.1093/imrn/rnv176},
       URL = {https://doi.org/10.1093/imrn/rnv176},
}

@misc{Li,
  author        = {Li, Mingxiang},
  title         = {On the positivity of {Yamabe} invariant and {Paneitz} operator},
  howpublished  = {arXiv:2608.09279},
  year          = {2026},
  eprint        = {2608.09279},
  archivePrefix = {arXiv},
  primaryClass  = {math.DG},
  doi           = {10.48550/arXiv.2608.09279},
  url           = {https://arxiv.org/abs/2608.09279},
}

@misc{GKW,
  author        = {Gong, Liuwei and Kim, Seunghyeok and Wei, Juncheng},
  title         = {Compactness and non-compactness theorems of the fourth-
                   and sixth-order constant {$Q$}-curvature problems},
  howpublished  = {arXiv:2502.14237},
  year          = {2025},
  eprint        = {2502.14237},
  archivePrefix = {arXiv},
  primaryClass  = {math.AP},
  doi           = {10.48550/arXiv.2502.14237},
  url           = {https://arxiv.org/abs/2502.14237},
}

@misc{GLW,
  author        = {Gong, Liuwei and Lee, Sanghoon and Wei, Juncheng},
  title         = {Global convergence of the {Gursky--Malchiodi}
                   {$Q$}-curvature flow},
  howpublished  = {arXiv:2602.04267},
  year          = {2026},
  eprint        = {2602.04267},
  archivePrefix = {arXiv},
  primaryClass  = {math.DG},
  doi           = {10.48550/arXiv.2602.04267},
  url           = {https://arxiv.org/abs/2602.04267},
}

@article {LX,
    AUTHOR = {Li, YanYan and Xiong, Jingang},
     TITLE = {Compactness of conformal metrics with constant
              {$Q$}-curvature. {I}},
   JOURNAL = {Adv. Math.},
  FJOURNAL = {Advances in Mathematics},
    VOLUME = {345},
      YEAR = {2019},
     PAGES = {116--160},
      ISSN = {0001-8708,1090-2082},
   MRCLASS = {53A30 (53C20)},
  MRNUMBER = {3899029},
MRREVIEWER = {Gang\ Li},
       DOI = {10.1016/j.aim.2019.01.020},
       URL = {https://doi.org/10.1016/j.aim.2019.01.020},
}

@article {Ligang,
    AUTHOR = {Li, Gang},
     TITLE = {A compactness theorem on {B}ranson's {$Q$}-curvature equation},
   JOURNAL = {Pacific J. Math.},
  FJOURNAL = {Pacific Journal of Mathematics},
    VOLUME = {302},
      YEAR = {2019},
    NUMBER = {1},
     PAGES = {119--179},
      ISSN = {0030-8730,1945-5844},
   MRCLASS = {53C21 (35B50 35J61 35R01 53A30)},
  MRNUMBER = {4028770},
MRREVIEWER = {Changwei\ Xiong},
       DOI = {10.2140/pjm.2019.302.119},
       URL = {https://doi.org/10.2140/pjm.2019.302.119},
}

@article {HR,
    AUTHOR = {Hebey, Emmanuel and Robert, Fr\'ed\'eric},
     TITLE = {Coercivity and {S}truwe's compactness for {P}aneitz type
              operators with constant coefficients},
   JOURNAL = {Calc. Var. Partial Differential Equations},
  FJOURNAL = {Calculus of Variations and Partial Differential Equations},
    VOLUME = {13},
      YEAR = {2001},
    NUMBER = {4},
     PAGES = {491--517},
      ISSN = {0944-2669,1432-0835},
   MRCLASS = {58J60},
  MRNUMBER = {1867939},
       DOI = {10.1007/s005260100084},
       URL = {https://doi.org/10.1007/s005260100084},
}

@misc{MP,
  author        = {Mazumdar, Saikat and Premoselli, Bruno},
  title         = {Compactness of conformal metrics with constant
                   {$Q$}-curvature of higher order},
  howpublished  = {arXiv:2510.00888},
  year          = {2025},
  eprint        = {2510.00888},
  archivePrefix = {arXiv},
  primaryClass  = {math.AP},
  doi           = {10.48550/arXiv.2510.00888},
  url           = {https://arxiv.org/abs/2510.00888},
}

@article {Juhl,
    AUTHOR = {Juhl, Andreas},
     TITLE = {Explicit formulas for {GJMS}-operators and {$Q$}-curvatures},
   JOURNAL = {Geom. Funct. Anal.},
  FJOURNAL = {Geometric and Functional Analysis},
    VOLUME = {23},
      YEAR = {2013},
    NUMBER = {4},
     PAGES = {1278--1370},
      ISSN = {1016-443X,1420-8970},
   MRCLASS = {53A30 (05A19 35R01 53A55 53B20 53C25 58J50)},
  MRNUMBER = {3077914},
MRREVIEWER = {A.\ Rod\ Gover},
       DOI = {10.1007/s00039-013-0232-9},
       URL = {https://doi.org/10.1007/s00039-013-0232-9},
}

@article {LiXu,
    AUTHOR = {Li, Mingxiang and Xu, Xingwang},
     TITLE = {On positivity of the {$Q$}-curvatures of conformal metrics},
   JOURNAL = {J. Funct. Anal.},
  FJOURNAL = {Journal of Functional Analysis},
    VOLUME = {289},
      YEAR = {2025},
    NUMBER = {8},
     PAGES = {Paper No. 111011, 23},
      ISSN = {0022-1236,1096-0783},
   MRCLASS = {35K91 (53C18)},
  MRNUMBER = {4896811},
MRREVIEWER = {Rong\ Zhang},
       DOI = {10.1016/j.jfa.2025.111011},
       URL = {https://doi.org/10.1016/j.jfa.2025.111011},
}

@misc{VZ,
  author        = {V{\'e}tois, J{\'e}r{\^o}me and Zeitler, Samuel},
  title         = {Positivity and non-positivity results for the sixth-order
                   {$Q$}-curvature of conformal metrics in {$\mathbb R^n$}},
  howpublished  = {arXiv:2607.18205},
  year          = {2026},
  eprint        = {2607.18205},
  archivePrefix = {arXiv},
  primaryClass  = {math.DG},
  doi           = {10.48550/arXiv.2607.18205},
  url           = {https://arxiv.org/abs/2607.18205},
}

@article {CG,
    AUTHOR = {Case, Jeffrey S. and Gover, A. Rod},
     TITLE = {The {GJMS} operators in geometry, analysis and physics},
   JOURNAL = {J. Lond. Math. Soc. (2)},
  FJOURNAL = {Journal of the London Mathematical Society. Second Series},
    VOLUME = {113},
      YEAR = {2026},
    NUMBER = {1},
     PAGES = {Paper No. e70375, 21},
      ISSN = {0024-6107,1469-7750},
   MRCLASS = {53C18 (32V05 35Q40 53C21 58J70)},
  MRNUMBER = {5011581},
       DOI = {10.1112/jlms.70375},
       URL = {https://doi.org/10.1112/jlms.70375},
}

@article {FWX,
    AUTHOR = {Fazly, Mostafa and Wei, Juncheng and Xu, Xingwang},
     TITLE = {A pointwise inequality for the fourth-order {L}ane-{E}mden
              equation},
   JOURNAL = {Anal. PDE},
  FJOURNAL = {Analysis \& PDE},
    VOLUME = {8},
      YEAR = {2015},
    NUMBER = {7},
     PAGES = {1541--1563},
      ISSN = {2157-5045,1948-206X},
   MRCLASS = {35J30 (35B08 35B45 35B50 35B65 35J91 35R45 53C21)},
  MRNUMBER = {3399131},
MRREVIEWER = {Yoichi\ Miyazaki},
       DOI = {10.2140/apde.2015.8.1541},
       URL = {https://doi.org/10.2140/apde.2015.8.1541},
}
\endgroup

\end{document}